\documentclass[a4paper,10pt]{amsart}
\usepackage[utf8x]{inputenc}

\usepackage{amsmath}
\usepackage{amsfonts} 
\usepackage{latexsym} 
\usepackage{amssymb}
\usepackage{amsthm}   
\usepackage{mathrsfs} 

\newcommand{\Q}{\mathbb{Q}}

\newcommand{\N}{\mathbb{N}}

\renewcommand{\phi}{\varphi}
\renewcommand{\epsilon}{\varepsilon}

\newtheorem{thm}{Theorem}
\newtheorem{lemma}[thm]{Lemma}
\newtheorem{prop}[thm]{Proposition}
\newtheorem{cor}[thm]{Corollary}

\theoremstyle{definition}

\theoremstyle{remark}
\newtheorem{rmk}{Remark}

\title{On the integral values of a curious recurrence}
\author[R. Dvornicich]{Roberto Dvornicich}
\address[R. Dvornicich]{Dipartimento di Matematica\\ Università di Pisa\\ Largo Pontecorvo 5\\ 56127 Pisa\\ Italia}
\email{dvornic@dm.unipi.it}
\author[F. Veneziano]{Francesco Veneziano}
\address[F. Veneziano]{Institut für Analysis und Computational Number Theory (Math A)\\ TU Graz\\ Steyrergasse 30/II\\ 8010 Graz\\ Österreich}
\email{veneziano@math.tugraz.at}
\author[U. Zannier]{Umberto Zannier}
\address[U. Zannier]{Scuola Normale Superiore\\ Piazza dei Cavalieri 7\\ 56126 Pisa\\ Italia}
\email{u.zannier@sns.it}
\begin{document}

\begin{abstract}
    We discuss a problem initially thought for the Mathematical Olympiad but which has several interpretations. The recurrence sequences involved in this problem may be generalized to recurrence sequences related to 
a much larger set of diophantine equations.
\end{abstract}


\maketitle

The purpose of this note is to comment on a problem shortlisted for the Romanian Mathematical 
Olympiad 2010 (see \cite[Shortlisted Problems for the $61^{th}$ NMO, No. 16]{RMC}). 


\smallskip

{\bf Problem} Let $x_0,x_1,x_2,\dots$ be the sequence defined by 
\begin{align}
 x_0&=1\notag\\
 \label{RicorrenzaXn}x_{n+1}&=1+\frac{n}{x_n}, &\forall n\geq 0.
\end{align}
What are the values of $n$ for which $x_n$ is an integer? 

\medskip


The author of this problem is Gheorghe  Iurea, but his solution does not appear in the above quoted booklet.

\medskip 

We have found this problem of interest, not only in itself, but also because {\it a posteriori} it may be dealt with in  different ways, each of which involves mathematical arguments of various nature. For instance, the second solution below uses a linear differential equation of the second order, which admits a solution which is a well-known function in combinatorics.

It may be that the arguments extend to cover a whole bunch of problems of a similar sort.

\bigskip

Let us now go back to the above sequence. Note that its first few values  are
\begin{align}\label{PrimiValori}
 1,1,2,2,\frac{5}{2},\frac{13}{5},\frac{38}{13},\frac{58}{19},\frac{191}{58},\frac{655}{191}\dotsc
\end{align}

\medskip

We can iterate the recursive formula \eqref{RicorrenzaXn}  to obtain $x_{n+2}=1+\frac{n+1}{1+\frac{n}{x_n}}=\frac{(n+2)x_n+n}{x_n+n}$;  in general,  $x_{n+k}$ may be expressed in two different ways: first, by a kind of continued fraction 
involving $x_n$ and the integers in $\{n,\ldots ,n+k-1\}$,    second,  by a linear fractional transformation in $x_n$, namely $x_{n+k}=M_{n,k}(x_n)=\frac{\alpha_{n,k}x_n+\beta_{n,k}}{\gamma_{n,k}x_n+\delta_{n,k}}$, for suitable integer coefficients depending on $n,k$; here $M_{n,k}$ is associated to the matrix 
$\begin{pmatrix}\alpha_{n,k} & \beta_{n,k}\\ \gamma_{n,k}&\delta_{n,k}\end{pmatrix}$.

These matrices  satisfy  the recurrence
\begin{equation*}
M_{n,k+1}=\begin{pmatrix}1 & n+k\\ 1 & 0\end{pmatrix}M_{n,k}.
\end{equation*}

The alluded continued fraction has not bounded length and is not periodic, and it seems not easy to find a simple formula for the general term $x_n$ or for the matrices $M_{n,k}$. Also, the standard tools using congruences do not seem to lead directly to a solution of  the stated question. Nevertheless we shall see a number of methods to answer it.

\bigskip

\section*{First solution}

We now present a solution which turns out to be essentially the same of Gheorghe's, which  he kindly sent us.

\medskip

If we define $f_n(x)=1+\frac{n}{x}$, we see that $x_{n+1}=f_n(x_n)$, so that one is led to study the dynamics of  the sequence of functions $f_n$; we note that in this solution the {\it arithmetic} comes into play only at the end, whereas one starts just by studying the dynamics from the real variable viewpoint (rather than a variable in $\Q$).

\medskip

Let us call $y_n=\frac{1+\sqrt{4n+1}}{2}$ the (positive) fixed point of $f_n$. Plainly we have that if $x<y_n$ then $f_n(x)>y_n$ and vice versa.

We can prove by induction that
\begin{lemma}
 For every $n\geq 4$ we have
\begin{equation}\label{DisugRadici}
 y_{n-1}=\frac{1+\sqrt{4n-3}}{2}<x_n<\frac{1+\sqrt{4n+1}}{2}=y_n.
\end{equation}
\end{lemma}
\begin{proof}
By a direct computation, we have  $\frac{1+\sqrt{13}}{2}<\frac{5}{2}<\frac{1+\sqrt{17}}{2}$, which establishes the basis of the induction. Assuming \eqref{DisugRadici} holds for $n$, by the previous remark we have that $y_n<x_{n+1}$, so we need only to prove that
\begin{equation*}
 1+\frac{n}{x_n}<y_{n+1}.
\end{equation*}
By the inductive hypothesis, it is enough to show that
\begin{align*}
 1+\frac{n}{y_{n-1}}&<y_{n+1}, \qquad i.e.,\\
1+\frac{2n}{1+\sqrt{4n-3}}&<\frac{1+\sqrt{4n+5}}{2},
\end{align*}
which is an elementary, though tedious, computation.
\end{proof}
\begin{rmk}
 For the  values of $n$ smaller than $4$, we have $x_3=y_2=x_2=2$ and $x_1=y_0=x_0=1$.
\end{rmk}

\medskip

Let us now assume that $x_n$ is an integer for some $n\geq 4$. From the lemma  we have
\begin{align*}
 \frac{1+\sqrt{4n-3}}{2}<&x_n<\frac{1+\sqrt{4n+1}}{2},\\
\sqrt{4n-3}<2&x_n-1<\sqrt{4n+1},\\
4n-3<(2&x_n-1)^2<4n+1.
\end{align*}
However the last inequalities are inconsistent modulo $4$.

Therefore we conclude that the only integral values of the sequence are $x_0,\dotsc,x_3$.

\bigskip

Essentially the same solution may be reached by a slightly different approach.

The same conclusion as before can be reached if we show  that $n-1<x_n^2-x_n<n$ for $n\ge 4$. 

We argue by induction. The inequalities are  verified by direct inspection for $n=4$, since $3< \frac{25}{4} - \frac{5}{2} <4$. 

Now let $a_n=x_n^2-x_n$, and assume that the inequalities hold  up to $n$. We may write $a_{n+1}$ as 
$$
a_{n+1} = x_{n+1}^2 -x_{n+1} = x_{n+1}(x_{n+1}-1) = \left(1 + \frac{n}{x_n}\right)\frac{n}{x_n} = \frac {n(x_n+n)}{x_n^2}. 
$$ 
By the induction hypothesis we have: 
$$ x_n^2<x_n+n \Longrightarrow a_{n+1} > n \frac {(x_n+n)}{x_n+n} = n $$
and 
$$ x_n^2> x_n+n-1 \Longrightarrow a_{n+1} < n \frac {x_n+n}{x_n+n-1} <n+1 $$
since $x_n>1$.

\section*{Second solution}
To study the sequence $(x_n)$ from an arithmetic point of view we define two integer sequences $(a_n),(b_n)$ by the recurrences 
\begin{align}
 a_0&=1\notag\\
 b_0&=1\notag\\
 \label{RicorrenzaAn}a_{n+1}&=a_n+nb_n,& &\forall n\geq 0\\
 \label{RicorrenzaBn}b_{n+1}&=a_n,& &\forall n\geq 0. 
\end{align}
Comparing \eqref{RicorrenzaAn} and \eqref{RicorrenzaBn} with \eqref{RicorrenzaXn} we see immediately that they satisfy $x_n=\frac{a_n}{b_n}$, so $a_n,b_n$ are the  numerator and denominator respectively in some fractional representation of $x_n$; however   $a_n,b_n$ {\it a priori} need not be coprime, so the said fraction can be possibly simplified.

We also see that $b_n$ may be eliminated from the recurrence to get
\begin{align}
 a_0&=1\notag\\
 a_1&=1\notag\\
 \label{RicorrenzaAA}a_{n+2}&=a_{n+1}+(n+1)a_n,  &\forall n\geq 0,
\end{align}
with $x_n=\frac{a_n}{a_{n-1}}$.

\medskip

Let us define $d_n=\gcd(a_n,a_{n-1})$; $d_n$ tells us how much the reduced denominator of $x_n$ differs from $a_{n-1}$. So, to obtain a lower bound for said denominator, we need a lower bound for $a_{n-1}$ and an upper bound for $d_n$.

\begin{rmk}\label{DDivide}
 By the recurrence \eqref{RicorrenzaAA} we see that $d_{n+1}|a_{n+2}$, and so $d_{n+1}|d_{n+2}$; this will be helpful  in establishing an upper bound for $d_n$.
\end{rmk}

\bigskip

A lower bound for $a_n$ is easily obtained 
as in the following lemma.
\begin{lemma}\label{L.lowerbound}
 For every $n\geq 0$ we have $a_n\geq\sqrt{n!}.$
\end{lemma}
\begin{proof}
 We argue by induction on $n\ge 0$. We check that $a_0=1=\sqrt{0!}$, $a_1=1=\sqrt{1!}$, and assuming the bound for $a_n$ and $a_{n+1}$ we get
\begin{multline*}
 a_{n+2}=a_{n+1}+(n+1)a_n\geq \sqrt{(n+1)!}+(n+1)\sqrt{n!}=\\
=\sqrt{(n+2)!}\frac{1+\sqrt{n+1}}{\sqrt{n+2}}=\sqrt{(n+2)!}\sqrt{1+\frac{2\sqrt{n+1}}{n+2}}\geq \sqrt{(n+2)!}. \qedhere
\end{multline*}
\end{proof}

\bigskip

To get an upper bound for $d_n$ we introduce the exponential generating
function of the sequence $(a_n)_{n\in\N}$, namely 
\begin{equation*}
 F(x)=\sum_{n=0}^\infty \frac{a_n}{n!}x^n.
\end{equation*}

We consider $F(x)$ merely as a formal power series, although one could prove that it converges for every complex $x$.

\medskip

From the recurrence on $(a_n)$ we can  obtain a differential equation
for $F$; in fact, we can multiply \eqref{RicorrenzaAA} by
$\frac{x^n}{n!}$ and sum it for $n\geq 0$; since clearly 
$F'(x)=\sum_{n=0}^\infty \frac{a_{n+1}}{n!}x^n$ and $F''(x)=\sum_{n=0}^\infty \frac{a_{n+2}}{n!}x^n$, 
we obtain that $F$ satisfies the conditions  
\begin{equation}\label{DifferentialF}
\begin{cases}
 F(0)&=1\\
 F'(0)&=1\\
 F''(x)&=(x+1)F'(x)+F(x).
\end{cases}
\end{equation}
The Cauchy problem \eqref{DifferentialF} may be solved (in the ring of formal power series) to get 
\begin{equation*}
F(x)=e^{x+\frac{x^2}{2}},
\end{equation*}
 and we can use this explicit form to get a formula for $a_n$. In fact
\begin{align}
 \sum_{n=0}^\infty \frac{a_n}{n!}x^n&=e^{x+\frac{x^2}{2}}=e^x e^{\frac{x^2}{2}}=\sum_{m=0}^\infty \frac{x^m}{m!} \sum_{s=0}^\infty \frac{x^{2s}}{2^s s!}\notag\\
 \frac{a_n}{n!}&=\sum_{2s+m=n}\frac{1}{2^s s!}\frac{1}{m!}\notag\\
 \label{FormulaAN}a_n&=\sum_{2s\leq n}\frac{n!}{2^s s! (n-2s)!}=\sum_{2s\leq n}\binom{n}{2s}(2s-1)!!,
\end{align}
where the {\it semifactorial}  $(2s-1)!!$ denotes as usual  the product $(2s-1)\cdot (2s-3)\dotsm 3\cdot 1$ and is defined to be $1$ for $s=0$.

\bigskip

\begin{rmk}
We note that these explicit formulas enable us to improve on Lemma \ref{L.lowerbound}. Indeed, using Stirling's formula, one may deduce that the `correct' order of magnitude of $\frac{a_n}{\sqrt{n!}}$ is  roughly $\exp(\sqrt n)$. A corresponding upper bound may be also obtained directly by induction, using the recurrence for $a_n$.

\end{rmk}

\bigskip

We can now use the preceding formula to prove the following lemma.
\begin{lemma}\label{Congruent1}
 Let $p$ be and odd prime. If $p|n$, then $a_n\equiv 1\pmod p.$ 
\end{lemma}
\begin{proof}
 Let $p$ be an odd prime dividing $n$. \\
If $p<2s$, we have that $p|(2s-1)!!$, as $p$ itself is one of the factors in the defining product of $(2s-1)!!$.\\
If $0<2s<p$ the binomial $\binom{n}{2s}$ is divisible by $p$, as the $p$ factor in $n$ is not cancelled by $(2s)!$.\\
So we have that in  formula \eqref{FormulaAN} only the term with $s=0$ is not divisible by $p$, whence 
\begin{equation*}
 a_n=\sum_{2s\leq n}\binom{n}{2s}(2s-1)!!\equiv 1 {\pmod p}.   \qedhere
\end{equation*}
\end{proof}
Applying  this lemma we get the following property of  $d_n$.
\begin{cor}\label{Corollario}
 For every $n\geq 1$, $d_n$ is a power of $2$.
\end{cor}
\begin{proof}
 If an odd prime $p$ divides $d_m$ for some $m\geq 1$, then, by Remark \ref{DDivide}, $p$ divides $d_n$ (and hence $a_n$) for all $n\geq m$, so also for $n=pm$. But this is not possible because $a_{pm}\equiv 1\pmod p$ by Lemma \ref{Congruent1}.
\end{proof}

We are now ready to prove an upper bound for $d_n$, which will follow  by using again
the exponential generating function $F(x)$.

\begin{prop}
For every $n\geq 1$ we have that $d_n\leq 2^{n-1}$.
\end{prop}
\begin{proof}  We have the following  identities concerning the above generating function $F(x)$:

\begin{equation*}
\left (\sum_{m=0}^\infty a_m \frac{x^m}{m!}\right)\left( \sum_{r=0}^\infty (-1)^r a_r\frac{x^r}{r!}\right)=F(x)F(-x)=e^{x^2}=\sum_{n=0}^\infty 
\frac{x^{2n}}{n!}.
\end{equation*}
Comparing the coefficients of $x^{2n}$ for any $n\geq 1$, we obtain
\begin{align}
\label{Potenzedi2}  \sum_{m+r=2n}(-1)^r\binom{2n}{m}a_ma_r= \frac{(2n)!}{n!}=2^n \cdot (2n-1)!!.
\end{align}
Now, as observed in Remark 3 above, $d_{n+1}$ divides any $a_m$ with $m\geq n$, so it divides the left-hand side of \eqref{Potenzedi2}, and we know from the Corollary \ref{Corollario} that it is a power of 2, therefore $d_{n+1}\leq 2^n$ for $n\geq 1$; of course $d_1=1= 2^0$.
\end{proof}

To get the conclusion, denote by $D_n$ the \emph{reduced} denominator of $x_n$. We have: 
\begin{equation*}
 D_n\geq \frac{a_{n-1}}{d_n}\geq\frac{\sqrt{(n-1)!}}{2^{n-1}}\qquad \forall n\geq 1.
\end{equation*}

It is easily seen that 
\begin{equation*}
 \sqrt{(n-1)!}>2^{n-1}\qquad \forall n\geq 10,
\end{equation*}
so we are left to inspect the values of $x_n$ with $0\leq n\leq 9$, which are exactly the values listed in \eqref{PrimiValori}.
\smallskip 

\begin{rmk} It is probably worth noting that the exponential
generating function $F(x)$ is widely known in the literature, and for
instance it can be interpreted as the exponential generating functions of
the number $a_n$ of involutions of the symmetric group $\mathcal{S}_n$ 
(see for instance \cite[Thm 3.16]{wilf}).
\end{rmk}

\section*{Further observations}
We can actually say much more about the numbers $d_n$.

\begin{prop}
If we define 
\begin{equation*}
 e_n=
\begin{cases}
 k &\text{if $n=4k$}\\
 k &\text{if $n=4k+1$}\\
 k+1 &\text{if $n=4k+2$}\\
 k+2 &\text{if $n=4k+3$,}
\end{cases}
\end{equation*}
then,  for every $n\geq 0$,  the exact power of $2$ dividing $a_n$ equals $2^{e_n}$.
\end{prop}

\begin{proof}
  Letting $q_n:=a_n/2^{e_n}$, we are reduced to prove that $q_n$ is an
  odd integer for $n\ge 0$. We need the following lemma.

\begin{lemma}
For every $n\geq 2$ we have
\begin{equation*}
 a_{n+6}=2(n^2+9n+19)a_{n+2}-n(n-1)(n+2)(n+5)a_{n-2}.
\end{equation*}
\end{lemma}
\begin{proof}   Indeed, since the sequence $(a_n)$ verifies a linear recurrence of the {\it second} order, any {\it three} sequences of the shape $(a_n), (a_{n+r}), (a_{n+s})$ are   linearly related by an equation with coefficients which are polynomials in $n$; they may be found by easy elimination. Presently, we are interested in the case $r=4,s=8$, 
where  this elimination is hidden in  the following explicit calculations:

\begin{align*}
a_{n+6} = & \ a_{n+5}+(n+5)a_{n+4}=(n+6)a_{n+4}+(n+4)a_{n+3}\\
= &\ (2n+10)a_{n+3}+(n+6)(n+3)a_{n+2}\\
= &\ (n^2+11n+28)a_{n+2}+2(n+5)(n+2)a_{n+1}\\
= &\ 2(n^2+9n+19)a_{n+2}-(n+2)(n+5)a_{n+2}+2(n+2)(n+5)a_{n+1}\\
= &\ 2(n^2+9n+19)a_{n+2}+(n+2)(n+5)a_{n+1}-(n+2)(n+5)(n+1)a_{n}\\
= &\ 2(n^2+9n+19)a_{n+2}-n(n+2)(n+5)a_{n}+n(n+2)(n+5)a_{n-1}\\
= &\ 2(n^2+9n+19)a_{n+2}-n(n-1)(n+2)(n+5)a_{n-2}. \qedhere
\end{align*}\end{proof}

\medskip

\begin{rmk} It is worth noticing that the relation with {\it polynomial} coefficients that we have obtained is `monic', in the sense that  the coefficient of $a_{n+6}$   is $1$. This feature, which is for us important, is not {\it a priori}  guaranteed for a linear recurrence with polynomial coefficients, and appears to us as a piece of good luck.

\end{rmk}

\medskip

Having proved the lemma, 
if we divide the relation in this last lemma by $2^{e_{n+6}}$ we obtain the recurrence
\begin{equation*}
 q_{n+6}=(n^2+9n+19)q_{n+2}-\frac{n(n-1)(n+2)(n+5)}{4}q_{n-2}.
\end{equation*}
Observe that $n^2+9n+19$ is odd for every $n$, while $\frac{n(n-1)(n+2)(n+5)}{4}$ is an even integer for every $n$: indeed, for $n$ even (resp. $n$ odd), the product $n(n+2)$ (resp. $(n-1)(n+5)$) is divisible by $8$.

So, after checking by inspection that 
\begin{equation*}
 q_0=1,q_1=1,q_2=1,q_3=1,q_4=5,q_5=13,q_6=19,q_7=29
\end{equation*}
are all odd integers, we obtain by induction that the $q_n$ are all odd integers.
\end{proof}
\bigskip


In view  of the previous definitions and by Corollary \ref{Corollario}, we can now compute $d_n$ by a `closed' formula: 
\begin{prop}
\label{dn}
\begin{equation*}
 d_n=
\begin{cases}
 2^k &\text{if $n=4k$}\\
 2^k &\text{if $n=4k+1$}\\
 2^k &\text{if $n=4k+2$}\\
 2^{k+1} &\text{if $n=4k+3$.}
\end{cases}
\end{equation*}
\end{prop}

As a final remark, we observe that Proposition \ref{dn} easily implies that:

\begin{enumerate}
\item [(a)] The estimate $d_n\leq 2^{\frac{n+1}{4}}$ holds for all $n>0$, which improves on the lower bound of the reduced denominator $D_n$ to
\begin{equation*}
  D_n\geq \frac{\sqrt{(n-1)!}}{2^\frac{n+1}{4}}\qquad \forall n\geq 1, 
\end{equation*}
and says directly that $D_n>1$ for $n\geq 4$; 

\item [(b)] the denominator $D_n \,(n>0)$ is even if and only if $n\equiv 0 \pmod 4$, while the corresponding numerator is even if and only if $n\equiv 2,3 \pmod 4$.
\end{enumerate}

\end{document}